\documentclass[12pt,thmsa]{article}
\usepackage{amsfonts}
\usepackage{amsmath}
\usepackage{dsfont}
\usepackage{amsthm}
\usepackage{amssymb, graphics, setspace}
\usepackage[useregional]{datetime2}
\usepackage{inputenc}
\newcommand{\mathsym}[1]{{}}
\newcommand{\unicode}[1]{{}}
\newcommand{\re}{\mathrm{e}}
\newcommand{\ri}{\mathrm{i}}
\newcommand{\rd}{\mathrm{d}}
\newtheorem{theorem}{Theorem}
\newtheorem{definition}{Definition}

\begin{document}

\title {Translation surfaces in the Heisenberg Group}
\author{Christiam B. Figueroa\thanks{%
Principal professor of the Pontificia Universidad Cat\'{o}lica del Per\'{u}}}


\maketitle

\begin{abstract}
A translation  surface in the Heisenberg group is constructed as the product of two planar curves. We classify a type of such surfaces with vanishing intrinsic curvature  by analyzing the determinant of their Gauss map .\\[0.2cm]
 \textsl{MSC}: {53A10; 53A35}\\
 \textsl{Keywords}:  Gauss map, Heisenberg group,  translation surface, intrinsic curvature.
\end{abstract}

\label{first}


\section{Introduction}

The aim of this paper is  to investigate the translation surfaces in the 3-dimensional  Heisenberg group, $\mathcal{H}_3$,  which are defined as the product of two planar curves. We seek to   characterize these surfaces based when their intrinsic curvature is zero. To achieve this, we  analyze the determinant of the surface´s Gauss map.

The paper is structured as follows.  Section 2 provides an overview of the Gans model of hyperbolic geometry. Section 3 summarizes the essential geometry of the Heisenberg group. In the fourth section, we study general non-parametric surfaces in $\mathcal{H}_3$ and calculate their  fundamental forms and curvature. Section 5 ie devoted to  the Gauss map for this type of surface,including the formula for the determinant of its differential. Finally, in the last section we establish the classification of minimal and flat non-parametric translation surfaces based on the determinant of their Gauss map.

\section{The Gans Model }
This is a  model of the hyperbolic geometry, developed by David Gans, see \cite{gans1966new}.
Consider the Poincar\'{e} Disk
$$\mathbb{D}=\{(x,y): x^{2}+y^{2}<1\}$$
endowed with the metric
$$g(x,y)=\frac{4}{(1-x^{2}-y^{2})^{2}}(\rd x^{2}+\rd y^{2}). $$

We  will define a a diffeomorphism between   the Poincar\'{e} disk and the plane $\mathcal{P}: z=1$

Using the stereographic projection from the south pole $(0,0,-1)$ of the unit sphere, $S^2$, we define the following diffeomorphism  $\varphi$ between the upper hemisphere $S_{+}^{2}$ onto  the disk, $\mathbb{D}\subset \mathbb{R}^3$
$$\varphi(x,y,z)=(\frac{x}{z+1},\frac{y}{z+1},0).
$$
Similarly, considering  the stereographic projection from the origin $(0,0,0)$ of $S^2$, we define a diffeomorphism $\psi$ of $S_{+}^{2}$ onto  the plane  $\mathcal{P}: z=1$,
\begin{equation}\label{psi}
  \psi(x,y,z)=(\frac{x}{z},\frac{y}{z},1).
\end{equation}
Then, $F(x,y)=\psi\circ\varphi^{-1}$ is a diffeomorphism  from the disk $\mathbb{D}$ onto $\mathcal{P}$, where
\begin{equation}\label{diff}
  F(x,y,0)=(\frac{2x}{1-x^{2}-y^{2}},\frac{2y}{1-x^{2}-y^{2}},1)
\end{equation}
and the inverse is given by
$$F^{-1}(u,v,1)=(\frac{u}{1+\sqrt{1+u^{2}+v^{2}}},\frac{v}{1+\sqrt{1+u^{2}+v^{2}}},0)$$
Then the metric induced on  $\mathcal{P}$ by $F$ is given by
$$h(u,v)=\frac{(1+v^{2})\rd u^{2}-2uv\rd u\rd v+(1+u^{2})\rd v^{2}}{1+u^{2}+v^{2}}$$
The Riemannian space $(\mathcal{P},h)$ is the Gans model of  the hyperbolic geometry.

\subsection{Isometries}

Consider the Poincar\'{e} disk  as the subset $\mathbb{D}=\{z\in  \mathbb{C}: |z|<1\}$
of the complex plane and the Gans model $\mathcal{P}=\{w:w\in \mathbb{C}\}$.
We know that the set of orientation-preserving isometries of the Poincar\'{e} Disk have the form,
$$\rho(z)=\re^{\ri\theta}\frac{z-a}{1-\overline{a}z}, \qquad a\in\mathbb{D}.$$
And all  isometries of $\mathbb{D}$ are composed of $\rho$ with complex conjugation, that is reflection at the real axis.
Therefore, the isometry group of the Gans model is
$$\mathrm{Iso}(\mathcal{P})=\{F\circ\rho\circ F^{-1}: \rho\in \mathrm{Iso}(\mathbb{D})\},$$
where $F$ is as in $(\ref{diff})$. I shall highlight two cases:

If $\rho(z)=\re^{\ri\theta}z$ , then  $F\circ\rho\circ F^{-1}(w)=\re^{\ri\theta}w$, that is, a rotation about the origin $(0,0)$ is an isometry of the hyperbolic space $\mathcal{P}$.

On the other hand, if $\rho(z)=\overline{z}$, then  $F\circ\rho\circ F^{-1}(w)=\overline{w}$  is the
 reflection across the $u$ axis. Since rotation about the origin is an isometry, a reflection across the line $au+bv=0$  is an isometry too.

\section{The Geometry of the Heisenberg Group }
The 3-dimensional Heisenberg group $\mathcal{H}_{3}$ is a two-step nilpotent Lie group. It has the following standard representation
 in $GL_{3}(\mathbb{R})$
 $$\left[
      \begin{array}{ccc}
        1 & r & t \\
        0 & 1 & s \\
        0 & 0 & 1 \\
      \end{array}
    \right]
$$
with $r,s,t\in \mathbb{R}$.

In order to describe a left-invariant metric on $\mathcal{H}_{3}$, we note that the Lie algebra $\mathfrak{h}_{3}$ of $\mathcal{H}_{3}$ is given by the
matrices
$$A=\left[
      \begin{array}{ccc}
        0 & x & z \\
        0 & 0 & y \\
        0 & 0 & 0 \\
      \end{array}
    \right] $$
    with $x,y,z$ real. The exponential map $\exp:\mathfrak{h}_{3}\rightarrow \mathcal{H}_{3}$ is a global diffeomorphism, and is given by
    $$\exp(A)=I+A+\frac{A^{2}}{2}=\left[
                                   \begin{array}{ccc}
                                     1 & x & z+\frac{xy}{2} \\
                                     0 & 1 & y \\
                                     0 & 0 & 1 \\
                                   \end{array}
                                 \right].$$
    Using the exponential map as a global parametrization, with the identification of the Lie algebra $\mathfrak{h}_{3}$ with $\mathbb{R}^{3}$ given by
$$(x,y,z)\longleftrightarrow \left[
      \begin{array}{ccc}
        0 & x & z \\
        0 & 0 & y \\
        0 & 0 & 0 \\
      \end{array}
    \right]$$
    the group structure of $\mathcal{H}_{3}$ is given by
\begin{equation}\label{pr}
(a,b,c)\ast (x,y,z)=(a+x,b+y,c+z+\frac{ay-bx}{2}).
\end{equation}
    From now on, modulo the identification given by $exp$, we consider $\mathcal{H}_{3}$ as $\mathbb{R}^{3}$ with the product given in (\ref{pr}). The Lie   algebra bracket, in terms of the canonical basis $\{e_{1},e_{2},e_{3}\}$ of $\mathbb{R}^{3}$, is given by
    \vspace{0.5cm}
    $$[e_{1},e_{2}]= e_{3}, \qquad   [e_{i},e_{3}]=0,$$
    \vspace{0.5cm}
 with $i=1,2,3.$ Now, using $\{e_{1},e_{2},e_{3}\}$ as the orthonormal frame at the identity, we have the following left-invariant metric $\rd s^{2}$ in $\mathcal{H}_{3}$
$$\rd s^{2}=\rd x^{2}+\rd y^{2}+(\frac{1}{2}y \rd x-\frac{1}{2}x\rd y+\rd z)^{2}.$$
And the basis of the orthonormal left-invariant vector fields is given by
    $$E_{1}=\frac{\partial}{\partial x}-\frac{y}{2}\frac{\partial}{\partial z},\qquad  E_{2}=\frac{\partial}{\partial x}+\frac{x}{2}\frac{\partial}{\partial z},\qquad
    E_{3}=\frac{\partial}{\partial z}\cdot$$
Then the Riemann connection of $\rd s^{2}$, in terms of the basis $\{E_{i}\}$, is given by

$$
\begin{array}{ccccc}
  \nabla_{E_{1}}E_{2} & =& \frac{1}{2}E_{3}  &=&  -\nabla_{E_{2}}E_{1} \\\\
  \nabla_{E_{1}}E_{3} & =& -\frac{1}{2}E_{2} &=&  \nabla_{E_{3}}E_{1} \\\\
  \nabla_{E_{2}}E_{3} & =& \frac{1}{2}E_{1}  &=&  \nabla_{E_{3}}E_{2}
\end{array}
$$
and $\nabla_{E_{i}}E_{i}=0$ for $i=1,2,3$.

Using the fact that an isometry of $\mathcal{H}_{3}$ which fix the identity, is an automorphism of $\mathfrak{h}_{3}$, it is possible to show that evert isometry of $\mathcal{H}_{3}$ is of the form $L\circ A$ where  $L$ is a left translation in $\mathcal{H}_{3}$ and $A$  is in one of the following forms
$$\begin{bmatrix}
    \cos\theta & -\sin\theta\;\;\, & 0 \\
    \sin\theta & \cos\theta &0 \\
    0 & 0 & 1 \\

  \end{bmatrix}
\qquad \textrm{or} \qquad  \begin{bmatrix}
     \cos\theta & \sin\theta & 0 \\
    \sin\theta & -\cos\theta \;\;\, &0 \\
    0 & 0 & -1\;\; \\
   \end{bmatrix}.
   $$

 That is, $A$ represent a rotation around the $z$-axis or a composition of the reflection across the plane $z=0$ and a reflection across a line $y=mx$ for some $m\in \mathbb{R}$ .

 \section{Surfaces in $\mathcal{H}_{3}$}

Let $S$ be a graph of a smooth function $f:\Omega \rightarrow \mathbb{R}$ where $\Omega $ is an open set of $\mathbb{R}^{2}$. We consider the
following parametrization of $S$
\begin{equation}\label{paramet}
 X\left( x,y\right) =( x,y,f( x,y)),\qquad (x,y)\in \Omega.
\end{equation}
A basis of the tangent space $T_{p}S$ associated to this
parametrization is given by
\begin{equation}
\begin{array}{ccccc}
X_{x} & = & \left( 1,0,f_{x}\right) & = & E_{1}+\left( f_{x}+\frac{y}{2}%
\right) E_{3} \\\\
X_{y} & = & \left( 0,1,f_{y}\right) & = & E_{2}+\left( f_{y}-\frac{x}{2}%
\right) E_{3}%
\end{array}
\label{basis}
\end{equation}%
\noindent
where $f_x$ and $f_y$ denote the partial derivatives of $f$, with respecto $x$ and $y$ respectvely. And the  unit normal vector of $S$  is given by
\begin{equation}
\eta \big( x,y\big) =-\bigg(\frac{f_{x}+\displaystyle\frac{y}{2}}{w}\bigg)
E_{1}-\bigg( \frac{f_{y}-\displaystyle\frac{x}{2}}{w}\bigg) E_{2}+\frac{1}{w}E_{3}
\label{normal}
\end{equation}%
where
\begin{equation}
w=\sqrt{1+\left( f_{x}+ \frac{y}{2}\right) ^{2}+\left( f_{y}-\frac{x}{2}\right) ^{2}}.
\label{w}
\end{equation}
Then the
coefficients of the first fundamental form of $S$  are given by%
\begin{equation}
\begin{array}{ccccl}
E & = & \langle X_{x}, X_{x}\rangle & = & 1+\left( f_{x}+ \displaystyle\frac{y}{2}\right) ^{2} \\\\
F & = & \langle X_{y},X_{x}\rangle & = & \left( f_{x}+\displaystyle\frac{y}{2}\right) \left( f_{y}-%
\displaystyle\frac{x}{2}\right) \\\\\
G & = & \langle X_{y},X_{y}\rangle & = & 1+\left( f_{y}- \displaystyle\frac{x}{2}\right) ^{2}.%
\end{array}
\label{1ffund}
\end{equation}%
If $\nabla $ is the Riemannian connection of $\left( \mathcal{H}%
_{3},\rd s^{2}\right) $, by the Weingarten  formula for hypersurfaces, we have that%
\[
A_{\eta }v=-\nabla _{v}\eta ,\qquad  v\in T_{p}S
\]%
and the coefficients of the second fundamental form are given by
\begin{equation}
\begin{array}{ccccl}
L & = & -\langle\nabla _{X_{x}}\eta ,X_{x}\rangle & = & \displaystyle\frac{f_{xx}+( f_{y}-\frac{x%
}{2})( f_{x}+\frac{y}{2})}{w} \\\\
M & = & -\langle\nabla _{X_{x}}\eta ,X_{y}\rangle & = &\displaystyle \frac{f_{xy}+\frac{1}{2}\left(
f_{y}-\frac{x}{2}\right) ^{2}-\frac{1}{2}\left( f_{x}+\frac{y}{2}\right) ^{2}%
}{w} \\\\
N & = & -\langle\nabla _{X_{y}}\eta ,X_{y}\rangle & = & \displaystyle\frac{f_{yy}-\left( f_{y}-\frac{x%
}{2}\right) \left( f_{x}+\frac{y}{2}\right) }{w}\cdot%
\end{array}
\label{2ffund}
\end{equation}

Recall that the mean curvature of any surface of $\mathcal{H}_3$ can bes expressed in terms of its first and second fundamental forms, given. a parametrization,
$$H=\frac{1}{2}(\frac{EN+GL-2FM}{EG-f^2}$$
When the surface is graph of a smooth function $f$, we replace the coefficients given in (\ref{1ffund}) and (\ref{2ffund}) into the mean curvature formula
$$\frac{(1+q^2)f_{xx}-2pqf_{xy}+(1+p^2)f_{yy}}{(1+p^2+q^2)^{3/2}}=2H,$$
where $p=f_x+y/2$ and $q=f_y-x/2$. In particular, when $H=0$ the equation of the minimal graph is given by
\begin{equation}\label{meq}
  (1+q^2)f_{xx}-2pqf_{xy}+(1+p^2)f_{yy}=0
\end{equation}

We finish this section by calculating the Gaussian curvature for  a non-parametric surface, that is, a surface which is a graph over the $xy$-plane. This formula is also presented  in \cite{bekkar1991exemples}.
\begin{theorem}\label{curv}
Let $S$ be a non-parametric surface in $\mathcal{H}_{3}$ given by $(x,y,f(x,y))$ with $(x,y)\in \Omega\subset \mathbb{R}^{2}.$ Then the Gauss curvature of $S$ is given by
\begin{align*}
  w^{4}K & =w^{2}(f_{xy}^{2}-f_{xx}f_{yy}-\frac{1}{4})-(1+q^{2})\bigg((f_{xy}+\frac{1}{2})^{2}-f_{xx}f_{yy}\bigg) \\
   & -(1+p^{2})\bigg((f_{xy}-\frac{1}{2})^{2}-f_{xx}f_{yy}\bigg)+pq(f_{yy}-f_{xx},)
\end{align*}

\noindent where $p,q$ and $w$ are defined by
$$
p=f_{x}+\frac{y}{2},\;\; q=f_{y}-\frac{x}{2},\;\;\text{and}\;\; w=\sqrt{1+p^{2}+q^{2}}.
$$
\end{theorem}
\begin{proof}
See \cite{bekkar1991exemples} and \cite{figueroa2012gauss}
\end{proof}
In particular, when $K=0$ the equation of the flat graph is given by:

\begin{equation}\label{fq}
  \begin{aligned}
0&=w^{2}(f_{xy}^{2}-f_{xx}f_{yy}-\frac{1}{4})-(1+q^{2})\bigg((f_{xy}+\frac{1}{2})^{2}-f_{xx}f_{yy}\bigg)\\
   &-(1+p^{2})\bigg((f_{xy}-\frac{1}{2})^{2}-f_{xx}f_{yy}\bigg)+pq(f_{yy}-f_{xx})
\end{aligned}
\end{equation}

\section{The Gauss Map}
Recall that the Gauss map is a function  from an oriented surface, $S\subset \mathbb{E}^{3}$, to the unit sphere in the Euclidean space . It associates to every point on the surface its oriented unit normal vector. Considering the Euclidean space as a commutative Lie group, the Gauss map is just the translation of the unit normal vector at any point of the surface to the origin, the identity element of $\mathbb{R}^{3}$. Reasoning in this way we define a Gauss map in the following form

\begin{definition}
Let $S\subset G$ be an orientable hypersurface of a n-dimensional Lie group $%
G,$ provided with a left invariant metric. The map%
\[
\gamma :S\rightarrow S^{n-1}=\left\{ v\in \tilde{g}: \left\vert v\right\vert
=1\right\}
\]%
where $\gamma \left( p\right) =\rd L_{p}^{-1}\circ \eta \left( p\right) $, $\tilde{g}$ the Lie algebra of $G$
and $\eta $ the unitary normal vector field of $S,$ is called the Gauss map of $S.$
\end{definition}

\noindent We observe that
$$\rd\gamma\left( T_{p}S\right) \subseteq T_{\gamma \left( p\right) }S^{n-1}
 =  \left\{ \gamma \left( p\right) \right\} ^{\perp }  =
\rd L_{p}^{-1}\left( T_{p}S\right),$$

\noindent therefore $\rd L_{p}\circ \rd\gamma \left( T_{p}S\right)
\subseteq T_{p}S$ .

Now we  obtain a local expression of the Gauss map $\gamma $. In fact, we consider the following sequence of maps
$$\phi:\Omega \overset{X}\longrightarrow X(\Omega)\subset \mathcal{H}_{3}\overset{\gamma}\longrightarrow S^{2}\overset{\psi}\longrightarrow \mathcal{P}$$
where, $X$ is a parametrization of $S$  and $\psi$ is given by $(\ref{psi})$.

When $S$ is the graph of a smooth function $f\left(x,y\right)$ with $(x,y)$ in a domain $\Omega \subset \mathbb{R}^{2} $.  Then

\begin{equation}\label{gmap}
  \phi(x,y)=\left(-(f_{x}+\frac{y}{2}),-(f_{y}-\frac{x}{2})\right)
\end{equation}
and the Jacobian matrix  of $\phi$ is
\begin{equation}\label{jacob}
\rd \phi_{(x,y)}=\left(
                   \begin{array}{cc}
                     -f_{xx} & -f_{xy}-1/2 \\
                     -f_{xy}+1/2 & -f_{yy} \\
                   \end{array}
                 \right).
\end{equation}

\noindent Notice that
\begin{equation}
\det  \rd \phi_{(x,y)} =f_{xx}f_{yy}-f_{xy}^{2}+\frac{1}{4} \label{rank}
\end{equation}%
and we will call this expression, the determinant of the Gauss map at the point $(x,y).$ If $\Omega=\mathbb{R}^{2}$, the greatest lower bound of the absolute value
of $\det\rd \phi_{(x,y)}$ is zero. This was proved by A. Borisenko and E.  Petrov in \cite{borisenko2011surfaces}.

We know that in the Euclidean case the differential of the Gauss
map is just the second fundamental form for surfaces in
$\mathbb{R}^{3},$ this fact can be generalized for hypersurfaces
in any Lie group. The following theorem, see \cite{ripoll1991hypersurfaces},
states a relationship between the Gauss map and the extrinsic
geometry of $S.$

\begin{theorem}
\label{gauss}Let $S$ be an orientable hypersurfaces of a Lie group. Then
$$
\rd L_{p}\circ \rd\gamma _{p}\left( v\right) =-\left( A_{\eta }\left( v\right)
+\alpha _{\bar{\eta}}\left( v\right) \right) ,\qquad v\in T_{p}S
$$
where $A_{\eta }$ is the Weingarten operator, $\alpha _{\bar{\eta}}\left(
v\right) =\nabla _{v}\bar{\eta}$ and $\bar{\eta}$ is the left invariant
vector field such that $\eta \left( p\right) =\bar{\eta}\left( p\right).$
\end{theorem}

As a consequence of this theorem we have the following result
\begin{theorem}
\label{vertical}There is no graph of a smooth function over $XY$ with constant Gauss map
\end{theorem}

\begin{proof}See \cite{figueroa2012gauss}
\end{proof}

To end this section, we study the effect of the isometries of the Heisenberg group $\mathcal{H}_{3}$ on the  Gauss map of a surface.

\begin{theorem}

  Let $S$ be a graph of a smooth function $f:\Omega \rightarrow \mathbb{R}$ where $\Omega $ is an open set of $\mathbb{R}^{2}$ and $\phi:\Omega\rightarrow \mathcal{P}$ its Gauss map, where $X(\Omega)=S$.
\begin{enumerate}
  \item If $\rho_{\theta}:\mathcal{H}_{3}\rightarrow \mathcal{H}_{3}$ is a rotation about the $z$ axis by an angle $\theta$, then the Gauss map of $\rho_{\theta} (S)$ is $r_{\theta}\circ \phi$, where $r_{\theta}:\mathcal{P}\rightarrow \mathcal{P}$ is a rotation about the origin by an angle $\theta$.
  \item If $\sigma:\mathcal{H}_{3}\rightarrow \mathcal{H}_{3}$ is a reflection across the line $ax+by=0$ compound with the reflection about the plan $z=0$  then the gauss map of $\sigma (S)$ is $\tau\circ \phi$, where $\tau:\mathcal{P}\rightarrow \mathcal{P}$ is a reflection across the line $-bx+ay=0$.
\item If $L:\mathcal{H}_{3}\rightarrow \mathcal{H}_{3}$ is a left translation,  then the gauss map of $L(S)$ is $ \phi. $
\end{enumerate}
\end{theorem}

\begin{proof}
See \cite{figueroa2012gauss}

\end{proof}

\section{Translation surfaces in the Heisenberg group}
First, we assume that
\begin{eqnarray}
\nonumber
   \alpha(x)&=& (x,0,u(x)),\;\;x\in I \\
\nonumber
  \beta(y) &=& (0,y,v(y),\;\;y\in J
\end{eqnarray}
are two regular parameterize curves in the planes $XZ$ and  $YZ$, respectively.
 Then a translation parameterize surface is given by,
 $$X(x,y)=\alpha(x) \ast \beta(y)=(x,y,u(x)+v(y)+\frac{xy}{2}).$$
 So we considere this surfaces as a graph of a function

 \begin{equation}\label{emt}
   f(x,y)=u(x)+v(y)+\frac{xy}{2},
 \end{equation}
where $(x,y)\in I\times J$. From(\ref{rank}), we deduce that the determinant of the Gauss map for this surface is given by
$$\triangle =u''(x)v''(y).$$
The following two subsections are devoted to the study of minimal and flat translation surfaces through the determinant of their Gauss map. Recall that there is no graph such that its Gauss map is constant, see theorem (\ref{vertical}).

\subsection{Minimal non-parametric translation surface}
Combining the equation of a minimal  graph (\ref{meq}) with (\ref{emt}), we obtain the equation of minimal surface of this type as follows
\[(1+v'(y)^2)u''(x)-v'(y)(u'(x)+y)+(1+(u'(x)+y)^2)v''(y)=0
\]
When the determinant of the Gauss map  of this surfaces vanishes , that is $u''(x)v''(y)=0$, we show in Theorem 5.7 of \cite{figueroa2012gauss},  that the minimal graph, up to rigid motions (translation and rotation), is given by

  \begin{equation}\label{eqm}
    f(x,y)= \frac{xy}{2}+\frac{C}{2}[y\sqrt{1+y^2}+\ln(y+\sqrt{1+y^2})]
  \end{equation}

In \cite{inoguchi2012minimal}, J. Inoguchi proved that for a minimal translation surface of this type, should have $u''(x)=0$ or $v''(x)=0$. Therefore, $\triangle$ must be equal to zero and consequently  the surface given in (\ref{eqm}) is the unique minimal translation surface of this type.
\subsection{Flat non-parametric translation surface}
Similarly, by combining the equation of a flat graph (\ref{fq}) with (\ref{emt}), we  obtain the following flat surface equation,
\begin{equation}\label{cin}
  1+v'(y)^2=u''(x)v''(y)+v'(y)(y+u'(x))(v''(y)-u''(x))
\end{equation}
In analogy with the minimal case, we study these surfaces according to its Gauss map.
\begin{enumerate}
  \item If the determinant of its Gauss map is zero, that is, $\triangle =u''(x)v''(y)=0$. We consider two cases:
\begin{enumerate}
  \item[a)] If $u''(x)=0$ and $v''(y)\neq 0$. Let $u'(x)=A$, the above expression may be written as
  \begin{equation}
    1+(v'(y))^2-\frac{y}{2}\frac{d}{dy}(v'(y))^2=\frac{A}{2}\frac{d}{dy}(v'(y))^2
  \end{equation}
  Set $r(y)=(v'(y))^2$, we obtain the following ordinary differential equation
  $$\frac{A+y}{2}r'(y)-r(y)-1=0$$
  Solving this equation, we obtain
  $$r(y)=C(A+y)^2-1,\;\;C>0$$
  Then
  $$v'(y)=\sqrt{C(A+y)^2-1}$$
  Solving this differential equation,
  \begin{equation}\nonumber
  v(y)=\frac{(A+y)\sqrt{C(A+y)^2-1})}{2}+\frac{\ln|\sqrt{C(A+y)^2-1}-\sqrt{C}(A+y)| }{2\sqrt{C}}+D
  \end{equation}
  Since $u(x)=Ax+B$, we have, up to a vertical translation, the following flat non-parametric surface,
  \begin{multline}\label{c0}
    f(x,y)= \frac{xy}{2}+Ax + \frac{(A+y)\sqrt{C(A+y)^2-1})}{2}+\\
    \frac{\ln|\sqrt{C(A+y)^2-1}-\sqrt{C}(A+y)| }{2\sqrt{C}}
  \end{multline}
 where $A\in \mathbb{R}$ and $C>0$.

  \item[b)] If $u''(x)\neq0$ and $v''(y)= 0$. In this case, let $v´(y)=A\neq 0$. Replacing in the flat surface equation (\ref{cin}),
  $$1+A^2=-(y+u'(x))u''(x)$$
  That is,
  $$1+A^2+yu''(x)=-u'(x)u''(x)$$
  which is impossible, since $x$ and $y$ are independent variables. If $A=0$ we also arrive at a contradiction.
\end{enumerate}

  \item If the determinant of its  Gauss is different from zero, that is $\triangle=u''(x)v''(y)\neq0$. Set
\begin{align}
\label{H}
  H(y)&= 1+v'(y)^2-yv''(y)v'(y) \\\nonumber \\
\label{F}
  F(x,y)&= -(yu''(x)+u'(x)u''(x))\\\nonumber\\
\label{G}
  G(x,y)&= u''(x)+u'(x)v'(y)
\end{align}

Substituting into (\ref{cin}) we have
$$H(y)=F(x,y)v'(y)+G(x,y)v''(y)$$

Differentiating with respect  to $x$,
$$0=F_xv'(y)+G_xv''(y)$$
where
\begin{equation}\label{rr}
  F_x=-(yu'''(x)+(u'(x)u''(x))')\;\;\text{and} \;\;G_x=u'''(x)+u''(x)v'(y)
\end{equation}

Note that $v''(y)\neq 0$, so  if we set

\begin{equation}
\nonumber
  r(y)=\frac{v'(y)}{v''(y)}
\end{equation}

we obtain that
$$G_x+F_xr(y)=0$$

Substituting (\ref{rr}) into the above equation we obtain

\begin{equation}\label{3}
  u'''(x)=u'''(x)yr(y)+[u''(x)u'(x)]'r(y)-u''(x)v'(y).
\end{equation}

We have the following cases:

\begin{enumerate}
  \item[a)] If
$$u'''(x)=0$$
It follows that, $u''(x)=A\neq 0$ and from (\ref{3}), we obtain the following equation,
$$Ar(y)=v'(y)$$
Since $v''(y)\neq 0$, we have
$$v''(y)=A$$
If we replace in the flat surface equation, we obtain
$$1+(Ay+B)^2=A^2$$
which is impossible.

  \item[b)] If
  $u'''(x)\neq 0.$

  Differentiating equation (\ref{3})  with respect to the variable $y,$

  \begin{equation}\label{oo}
    u'''(x)(yr)'+[u''(x)u'(x)]'r'(y)-u''(x)v''(y)=0
  \end{equation}
  Since $\triangle=u''(x)v''(y)\neq 0$,  the above equation is equivalent to
  \begin{equation}\label{xx}
    v''(y)=\frac{u'''(x)}{u''(x)}(yr)'+\frac{[u''(x)u'(x)]'}{u''(x)}r'(y)
  \end{equation}

Differentiating (\ref{xx}) with respect $x$ and  by the independence of the  variables $x$ and $y$, it follows that
\begin{equation}\label{vy}
  \frac{(yr)'}{r'}=A
\end{equation}

Similarly, the equation (\ref{oo}) is equivalent to

\begin{equation}\label{yy}
  u''(x)=u'''(x)\frac{(yr)'}{v''}+[u'u'']'\frac{r'}{v''}
\end{equation}
Differentiation with respect $y$ we obtain
\begin{equation}\label{vx}
  [u'u'']'=Bu'''.
\end{equation}
$$$$
It easy to see that $B=-A$. Finally if we replace (\ref{vx}) and (\ref{vy}) in equation (\ref{oo}), we obtain
$$Au'''r'-Au'''r'-u''(x)v''(y)=0,$$
which is absurd, because $\triangle\neq0$

\end{enumerate}

\end{enumerate}

\noindent Consequently, we can state the following theorem:

\begin{theorem}
  Let the curves $\gamma_1$ and $\gamma_2$ be given by $\gamma_1(x)=(x,0,u(x))$ and $\gamma_2(y)=(0,y,v(y))$, respectively. The translation surface $S=\gamma_1*\gamma_2$, is a non-parametric surface given by $(x,y,f(x,y))$, where
  $$f(x,y)=\frac{xy}{2}+u(x)+v(y).$$
  Then, the determinant of its Gauss map is cero and

   \begin{multline}
   \nonumber
    f(x,y)= \frac{xy}{2}+Ax + \frac{(A+y)\sqrt{C(A+y)^2-1})}{2}+\\
    \frac{\ln|\sqrt{C(A+y)^2-1}-\sqrt{C}(A+y)| }{2\sqrt{C}}
  \end{multline}
  
\end{theorem}

If the constant $A$  vanishes, the result corresponds to the flat translation invariant surface previously obtained by J. Inoguchi \cite{inoguchi2005flat}.

Special thanks to Arshi Yousuf for his helpful observation, which helped improve the results of this paper.

 Christiam Figueroa \\
Departamento de Ciencias \\
Secci\'{o}n de Matem\'{a}ticas\\
Pontificia Universidad Cat\'{o}lica del Per\'{u}\\
Lima, Per\'{u}\\
{\it E-mail address}: {\tt cfiguer@pucp.pe}\\[0.3cm]

\label{last}

\end{document}